\documentclass[doublespcae]{tCON2e}

\everymath{\displaystyle}
\newtheorem{define}{Definition}
\newtheorem{assump}{Assumption}
\begin{document}
\markboth{International Journal of Control}{International Journal of Control}
\articletype{RESEARCH ARTICLE}
\title{Observer-Based Higher Order Sliding Mode Control of
Unity Power Factor in Three-Phase AC/DC Converter
for Hybrid Electric Vehicle Applications}
\author{Jianxing Liu, Salah Laghrouche$^*$, and Maxime Wack \\
\thanks{$^\ast$Corresponding author. Email: salah.laghrouche@utbm.fr
\vspace{6pt}} {\em{Laboratoire IRTES, Universit\'e de Technologie de Belfort-Montb\'eliard (UTBM), Belfort, France}}}
\maketitle
\begin{abstract}
In this paper, a full-bridge boost power converter topology is studied
for power factor control, using output higher order sliding mode control.
The AC/DC converters are used for charging the battery and super-capacitor
in hybrid electric vehicles from the utility.
The proposed control forces the input currents to track the desired values, which can
controls the output voltage while keeping the power factor close to one.
Super-twisting sliding mode observer is employed to estimate the input currents
and load resistance only from the measurement of output voltage.
Lyapunov analysis shows the asymptotic convergence of
the closed loop system to zero.
Simulation results show the effectiveness and robustness of the proposed controller.
\end{abstract}\bigskip
\bigskip
\section{INTRODUCTION}
With the advent of distributed DC power sources in the energy sector,
the use of boost type three phase rectifiers has increased in industrial applications,
especially, battery charger in hybrid electric vehicles (HEV) \citep{4140625,6179537,6031928,6380611}.
Power-factor-corrected utility interfaces are of great importance in the HEV industry.
The complete energy conversion cycle of the HEV must convert electrical power from the utility
to mechanical power at the drive axle as efficiently and as economically as possible \citep{4140625,6191321,6258881}.
Different power conversion systems of plug-in HEV power conditioning
systems are presented in \cite{5289744,5170004,5613206,5075550,5256209}.

Fig. \ref{fig:systemModel} shows the structure of the hybrid electric vehicle
power conversion system, which consists of of an AC/DC converter, a three phase
DC/AC inverter \citep{6151132,6248697}, three  DC/DC converters
and various power storages i.e. source grid, battery, super-capacitor and fuel cell.
The AC/DC converter is used to charge the battery and super-capacitor through
its bidirectional DC/DC converters, while ensuring
that the utility current is drawn at unity power factor in order to
minimize line distortion and maximize the real power available from the utility outlet.
The battery and super-capacitor supply power to the three-phase inverter which
feeds the three-phase motor.

The AC/DC converter consists of two stages \citep{4493430,777198}.
The first stage is Power Factor Correction (PFC), which simultaneously
regulates the DC-link voltage level and the line current waveform.
The second stage is a charger
with different types of resonant or pulse-width-modulation (PWM)
DC/DC converters \citep{4140625}.
PFC is used to improve the quality
of the input phase current that is sourced from the utility
by generating and tracking the desired current profile while the charger
is used to charge the battery and super-capacitor in a HEV.

Many schemes and solutions are proposed in the field of PFC. Linear control methods using
linear regulators for the output voltage control have been proposed in \cite{212395,9172},
which change the modulation index slowly, thus resulting in a slow dynamical response.
Consequently, the linear feedback control of the rectifier output voltage becomes slow and difficult.
Moreover, due to coupling between the duty-cycle and the state variables in the AC/DC boost converter,
linear controllers are not able to perform optimally for the whole range of operating conditions.
In contrast with linear control, nonlinear approaches can optimize the performance
of the AC/DC converter over a wide range of operating conditions.
Many nonlinear techniques have been proposed,
such as input-output linearization \citep{1187319},
feedback linearization \citep{845058}, fuzzy logic control \citep{1413544},
passivity-based control \citep{930975},
back-stepping technique control \citep{allag2007tracking},
Lyapunov-based control \citep{6179537,komurcugil1998lyapunov},
differential flatness based control \citep{6031928,6032095,6127921},
and sliding mode control \citep{unitypowerfactor,767067,4384358}.
However, most of the above works need continuous measurements of AC voltages,
AC currents and DC voltage.
This requires a large number of both voltage and current sensors,
which increases system complexity, cost, space, and reduces system reliability.
Moreover, the sensors are susceptible to electrical noise,
which cannot be avoided during high-power switching.
Reducing the number of sensors has a significant affect upon
the control system's performance.
A few results have been proposed to reduce the current sensors
\citep{andersen1999active,212395,1158977,925576} where the input phase currents are reconstructed
from the switching states of the AC/DC rectifier and the measured DC-link currents,
and then used in feedback control.
However, they require digital sampling of the DC-link current in every switching
cycle and numerical computations.
The accuracy of measurement is inherently controlled by the sampling rate.

The objective of this paper is to design an efficient AC/DC power converter
that charges the battery and super-capacitor in a HEV
with unity power factor, by eliminating the using of current sensors.
Only voltage sensors are required for measuring the output voltage and source voltage.
A Super-Twisting (ST) Sliding Mode Observer (SMO) is designed to observe the phase currents
and load resistance \citep{unitypowerfactor,6313905} from the measured output voltage.
The proposed ST SMO guarantees fast convergence rate
of the observation error dynamics, facilitating the design of controllers.
The controller and observer design
is based on the two goals mentioned in \cite{acdc_converter,edwards1998sliding},
for the design of an efficient AC/DC power converter:
1) Unity power factor to maximize the performance of the power conversion.
2) Ripple free output voltage.

Sliding mode algorithm is known for the characteristics of
robustness and effectiveness \citep{edwards1998sliding},
making it an effective method to deal with the nonlinear behavior of the boost rectifiers.
The ST Sliding Mode Control (SMC) allows not only the achievement of the high performance of the system but also the maintenance of
the functionality under parametric uncertainty and external disturbance.
A strong Lyapunov function is introduced to prove the stability of both the
observer and controller system.

The paper is organized as follows.
In Section \textrm{II}, the mathematical model and control objectives are presented.
In Section \textrm{III}, the design of observer-based current controller based on
ST Algorithm (STA) is presented.
In Section \textrm{IV}, we show the design of the parameter observer for the system,
and power factor is also estimated.
In Section \textrm{V}, simulations results of the performance of the obtained
ST SMC compared with the conventional PI controller are presented.
Finally, some conclusions are drawn in Section \textrm{VI}.
\begin{figure}[htbp]
  \centering
  \includegraphics[width=4.8in]{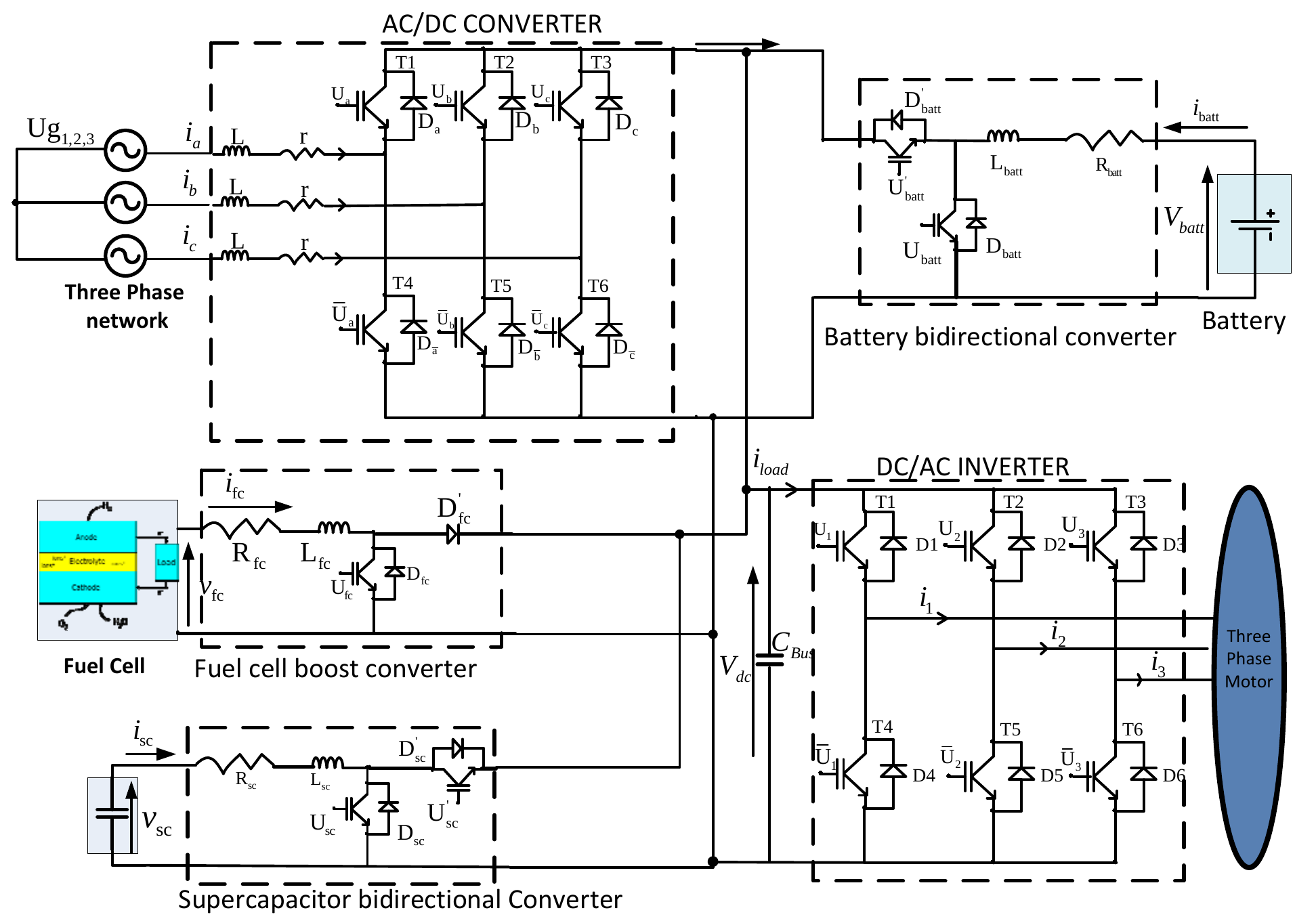}
  \caption{PEMFC powered hybrid system consists of Source Grid, Fuel Cell,
  Super-capacitor and battery}
  \label{fig:systemModel}
\end{figure}
\section{PROBLEM FORMULATION}\label{sec:Problemformulation}
\subsection{System Modeling}\label{sec:Model}
The power circuit of the three phase voltage source AC/DC full-bridge boost converter
under consideration is shown in Fig. \ref{fig:acdc}.
It is assumed that a equivalent resistive load
$R_L$ is connected to the output of the AC/DC converter \citep{6179537,6031928}.
\begin{figure*}[htbp]
  \centering
  \includegraphics[width=4.8in]{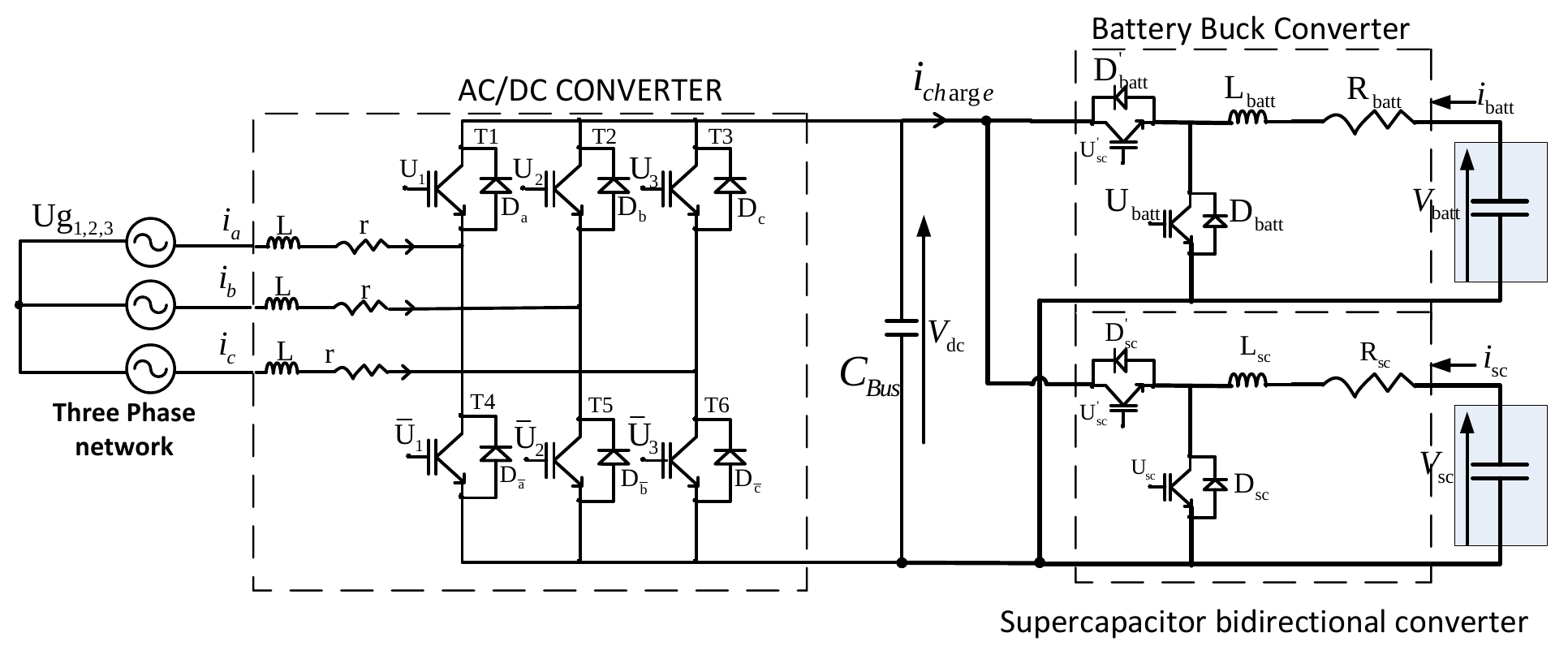}
  \caption{Electrical circuit of the three phase AC/DC boost converter}
  \label{fig:acdc}
\end{figure*}
The control inputs, as they appear in the system, are defined as
$\bm{u}=[u_1 \ \ u_2 \ \ u_3]^T$,
which take values from the discrete set $\{-1,+1\}$.
The corresponding inverse control $\bm{\bar{u}}=[\bar{u}_1 \ \ \bar{u}_2 \ \ \bar{u}_3]^T$
takes the opposite values at the same time,
i.e. $u_1=+1$ means $\bar{u}_1=-1$ which corresponds to the conducting state
for the upper switching element $T_1$ and nonconducting state
for the bottom switching element $T_4$ \citep{unitypowerfactor}.

The mathematical model of the boost AC/DC converter
in phase coordinate frame can be obtained
through analyzing the circuit \citep{212395},
\begin{equation}\label{eqn:dynamicmodel}
\left\{
\begin{split}
\frac{di_a}{dt}=&-\frac{r}{L}i_a-\frac{U_0}{6L}(2u_1-u_2-u_3)+\frac{1}{L}U_{g_1}, \\
\frac{di_b}{dt}=&-\frac{r}{L}i_b-\frac{U_0}{6L}(2u_2-u_1-u_3)+\frac{1}{L}U_{g_2}, \\
\frac{di_c}{dt}=&-\frac{r}{L}i_c-\frac{U_0}{6L}(2u_3-u_1-u_2)+\frac{1}{L}U_{g_3}, \\
\frac{dU_0}{dt}=&-\frac{U_0}{R_LC}+\frac{1}{2C}(i_au_1+i_bu_2+i_cu_3).
\end{split}
\right.
\end{equation}
It can also be written as,
\begin{equation}
\left\{
\begin{split}
\frac{\bm{di}}{\bm{dt}}=&-\frac{r}{L}\bm{i}-\frac{U_0}{6L}B\bm{u}+\frac{1}{L}\bm{U_g},\\
\frac{dU_0}{dt}=&-\frac{U_0}{R_LC}+\frac{1}{2C}\bm{u}^T\bm{i}.
\end{split}
\right.
\end{equation}
where $r$ is parasitic phase resistance
(including voltage source internal resistance and impedance of switching elements in open state);
$R_L$ is the load resistance;
$L$ is phase inductor;
$C$ is output capacitor;
$U_0$ is output voltage;
$\bm{i}=
\begin{bmatrix}
i_a&i_b&i_c
\end{bmatrix}
^T$ are the input phase currents;
$\bm{U_g}=
\begin{bmatrix}
U_{g_1}&U_{g_2}&U_{g_3}
\end{bmatrix}
^T$ are the source voltages which have different magnitudes but the same frequency and phase
shift of $\frac{2\pi}{3}$ electrical degrees (with respect to each other);
and
$\bm{u}=
\begin{bmatrix}
u_1&u_2&u_3
\end{bmatrix}
^T$ are control signals.
The gain matrix and source voltage are as follows,
\begin{equation}\label{gain matrix}
B=
\begin{bmatrix}
    2&-1&-1\\-1&2&-1\\-1&-1&2
\end{bmatrix},\ \
\bm{U_g}=E
\begin{bmatrix}
   sin(\theta)\\sin(\theta-\frac{2}{3}\pi)\\sin(\theta+\frac{2}{3}\pi)
\end{bmatrix}
\end{equation}
where $E$ is the magnitude of the source voltages \citep{komurcugil1998lyapunov}.

For modeling and control design, it is convenient to transform three-phase variables
into a rotating $(d,q)$ frame. The transformed variables is defined as,
\begin{equation}\label{transformed variables}
\begin{split}
\bm{u_{dq}}=&
\begin{bmatrix}
u_d\\u_q
\end{bmatrix}
=T\bm{u},\ \
\bm{i_{dq}}=
\begin{bmatrix}
i_d\\i_q
\end{bmatrix}
=T\bm{i},\\
T\bm{U_g}=&
\begin{bmatrix}
U_{g_d}\\U_{g_q}
\end{bmatrix}.
\end{split}
\end{equation}
where
\begin{equation}\label{eqn:park transofrmation}
T=\frac{2}{3}
\begin{bmatrix}
cos(\omega t)&cos(\omega t-\frac{2}{3}\pi)&cos(\omega t+\frac{2}{3}\pi)\\
sin(\omega t)&sin(\omega t-\frac{2}{3}\pi)&sin(\omega t+\frac{2}{3}\pi)
\end{bmatrix}
\end{equation}
is the Park's transformation \cite{Bose2002,1187319}.

From (\ref{gain matrix}) and (\ref{eqn:park transofrmation}),
it follows that $U_{g_d}=0$ and $U_{g_q}=E$.
The dynamical model of the AC/DC
converter in the rotating $(d,q)$ frame can be
expressed as \citep{komurcugil1998lyapunov,845058,648960}
\begin{equation}\label{eqn:dqmodel}
\left\{
\begin{split}
\frac{di_d}{dt} =& -\frac{r}{L}i_d+\omega i_q-\frac{U_{0}}{2L}u_d,\\
\frac{di_q}{dt} =& -\frac{r}{L}i_q+\frac{E}{L}-\omega i_d-\frac{U_{0}}{2L}u_q, \\
\frac{dU_{0}}{dt} =& -\frac{U_{0}}{R_LC}+\frac{3(i_du_d+i_qu_q)}{4C}.
\end{split}
\right.
\end{equation}
where $\omega$ is the angular frequency of the source voltage. In the transformed state equation (\ref{eqn:dqmodel}),
the state vector is defined as
$\bm{x} =
\begin{bmatrix}
x_1&x_2&x_3
\end{bmatrix}
^T=
\begin{bmatrix}
i_d&i_q&U_0
\end{bmatrix}
^T$ and the control input vector
$\bm{u_{dq}} =
\begin{bmatrix}
u_d&u_q
\end{bmatrix}
^T$
are the switching functions
$\bm{u} =
\begin{bmatrix}
u_1&u_2&u_3
\end{bmatrix}
^T$ in synchronously rotating $(d,q)$ coordinate.
From the control point of view, the model of AC/DC converter in $(d,q)$ frame has the advantage
of reducing the current control task into a set-point tracking problem \citep{1187319}.
\subsection{Control Objectives}\label{sec:Objectives}
\begin{assump}\label{assump:1}
The phase voltage $\bm{U_g}$ and output voltage $U_0$ are measurable;
\end{assump}

The control objectives are as follows,
\begin{itemize}
  \item The input phase currents $i_a,i_b,i_c$ should be in phase with corresponding
  input source voltage $U_{g_1},U_{g_2},U_{g_3}$ in order to obtain a unity power factor.
  \item The DC component of the output voltage should be
  driven to some desired value ${U^{*}_0}$ while its AC component has
  to be attenuated to a given level.
\end{itemize}
\section{OBSERVER-BASED SLIDING MODE CONTROLLER DESIGN}\label{sec:ProblemSolve}
In observer-based sliding mode control, the real plant states are substituted by observer
states, reducing the number of measurements.
It has been shown that the performance of an observer-based sliding mode
controller can be improved significantly by keeping the plant system and
the observer system operating closely \citep{573474}.
Fig. \ref{fig:controlstructure} shows the structure of
the observer-based control system for three phase AC/DC converters,
which consists of two important parts: sliding mode current observer and controller system.
\begin{figure}[htpb]
  \centering
  \includegraphics[width=4.4in]{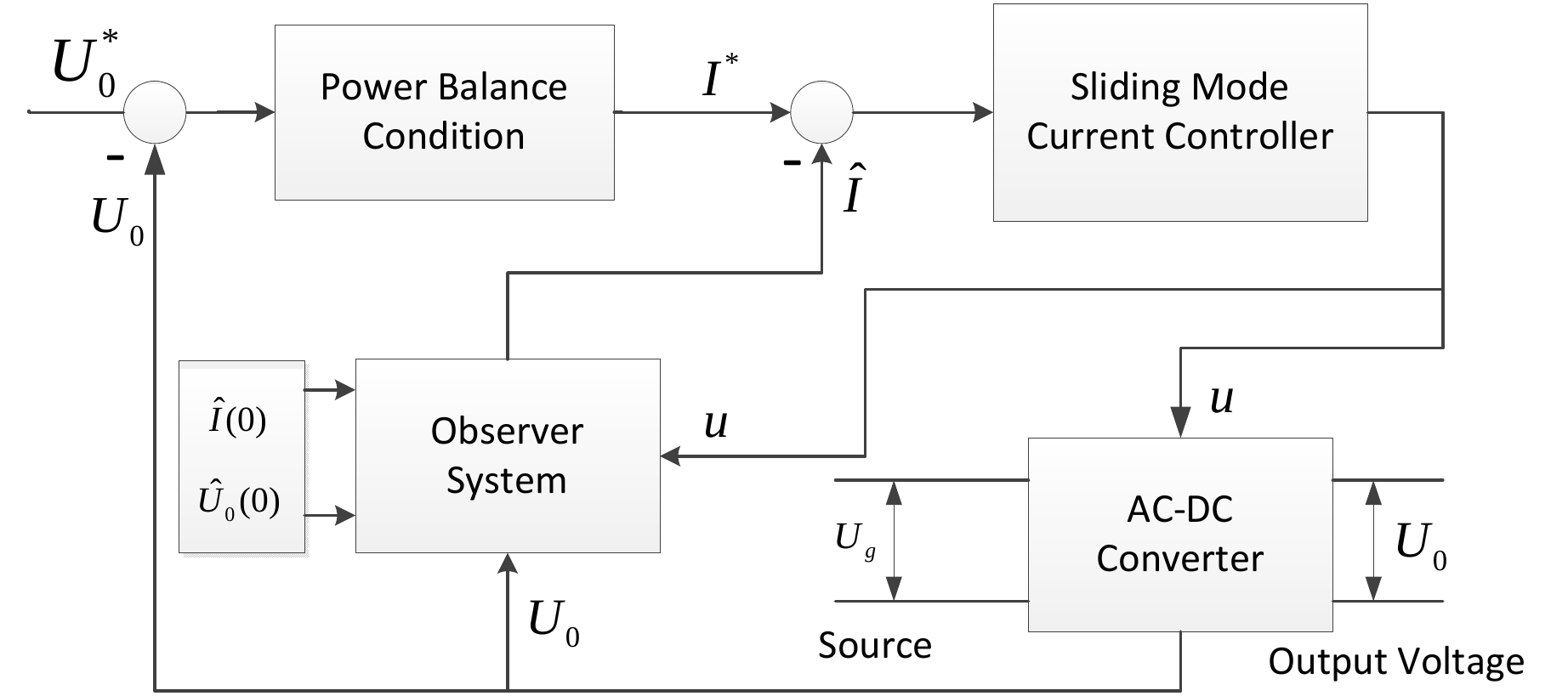}
  \caption{Observer-Based Control Structure of a Three Phase AC/DC Converter}
  \label{fig:controlstructure}
\end{figure}
\subsection{Super-Twisting Sliding Mode Observer Design}\label{subsec:sliding mode observer}
The proposed observer is designed in the following two steps,
\begin{enumerate}
  \item Analyzing the observability of the nonlinear system;
  \item Construction of the ST Observer;
\end{enumerate}
\subsubsection{Observability Analysis of the System}
In order to construct an observer for a system, it is necessary to verify the observability
of the system i.e. there exists the possibility of obtaining
the states of a system only from the knowledge of
its inputs and outputs up to time $t$ \citep{besancon2007nonlinear}.
Considering the following nonlinear system,
\begin{equation}\label{eqn:nonlinear system}
\left\{
\begin{split}
\dot{\bm{x}}=&\bm{f_u}(\bm{x},\bm{u}),\\
\bm{y}=&\begin{bmatrix}
h_1(x)&\cdots&h_p(x)
\end{bmatrix}
^T.
\end{split}
\right.
\end{equation}
where $\bm{x}\in \mathbb{R}^{n}$ are the state vectors, $\bm{u}\in \mathbb{R}^{m}$ are the bounded inputs, $\bm{y}\in \mathbb{R}^{p}$ are the outputs. Assume that the vector field $\bm{f_u}(\cdot,\cdot)$ is a sufficiently smooth function.
\begin{define}\label{define:observabilitycondition}\citep{hermann1977nonlinear}
The system described by (\ref{eqn:nonlinear system}) is locally observable
if the matrix defined by (\ref{matrix:observalibity}) satisfies
observability rank condition $dim(\mathrm{O})=n$ at a point $x_0$,
\begin{equation}\label{matrix:observalibity}
\begin{split}
\mathrm{O}=&
\begin{bmatrix}
    dL^0_{f_u}(h_1)&dL^0_{f_u}(h_2)&\cdots&dL^0_{f_u}(h_p)\\
    dL^1_{f_u}(h_1)&dL^1_{f_u}(h_2)&\cdots&dL^{1}_{f_u}(h_p)\\
    \vdots&\vdots&\ddots&\vdots\\
    dL^{n-1}_{f_u}(h_1)&dL^{n-1}_{f_u}(h_2)&\cdots&dL^{n-1}_{f_u}(h_p)
\end{bmatrix}
,
\end{split}
\end{equation}
where $L_{f_u}(h)$ denotes the Lie derivative of $h$ with respect to $f_u$.
\end{define}

Taking the output as $y=x_3=U_0$, the application of
Definition \ref{define:observabilitycondition}
leads to the following observability matrix,
\begin{equation}\label{matrix:observabilitymatrix}
\mathrm{O}=
\begin{bmatrix}
    0&0&1\\
    \frac{3u_d}{4C}& \frac{3u_q}{4C}&-\frac{1}{RC}\\
    -\varrho u_d-\frac{3\omega u_q}{4C}& \frac{3\omega u_d}{4C}-\varrho u_q&-\frac{3\|u_{dq}\|^2}{8LC}+(\frac{1}{R_LC})^2
\end{bmatrix}
,
\end{equation}
where $\varrho=\frac{3}{4C}(\frac{r}{L}+\frac{1}{R_LC})$, and $\|u_{dq}\|^2_2=u^2_d+u^2_q$.
Thus, it is possible to observe the currents $i_d,i_q$ from the measurement
of the output voltage $U_0$ when $ \|u_{dq}\|_2\neq 0$.
In the case of singular inputs $u_d=0,u_q=0$ which results the system (\ref{eqn:dqmodel})
into a reduced system with detectability property.
This property allows to construct an open-loop observer \citep{besancon2007nonlinear,sarinana2000nonlinear}.
\subsubsection{Construction of the Nonlinear Observer}
Assume that the only measured variable is the output DC voltage,
i.e. $y=x_3=U_0$. Thus, a ST SMO
for (\ref{eqn:dqmodel}) is constructed as follows
\begin{equation}\label{eqn:dqobservermodel}
\left\{
\begin{split}
\frac{d\hat{i}_d}{dt} =&-\frac{r}{L}\hat{i}_d + \omega \hat{i}_q-\frac{{U}_{0}}{2L}u_d+ k_1\mu(e_3),\\
\frac{d\hat{i}_q}{dt} = &-\frac{r}{L}\hat{i}_q-\omega \hat{i}_d-\frac{{U}_0}{2L}u_q+\frac{E}{L}
+k_2\mu(e_3),\\
\frac{d\hat{U}_{0}}{dt} =& -\frac{U_{0}}{R_LC}+\frac{3}{4C}(\hat{i}_du_d+\hat{i}_q u_q)+\mu(e_3).
\end{split}
\right.
\end{equation}
where the observation errors and STA are defined as
\begin{equation}\label{eqn:stw1}
\begin{split}
e_1 &= i_d-\hat{i}_d, e_2 = i_q-\hat{i}_q, e_3 = U_{0}-\hat{U}_{0},\\
\mu(e_3)&=\lambda |e_3|^{\frac{1}{2}}sign(e_3)+\alpha
\int_0^t sign(e_3)d\tau,
\end{split}
\end{equation}
with some positive constants $\lambda, \alpha$ respectively.

Then the error dynamics is given by
\begin{eqnarray}
\dot{e}_1 &=& -\frac{r}{L}e_1+\omega e_2-k_1\mu(e_3),
\label{eqn:dynamicobservererror1}\\
\dot{e}_2 &=& -\omega e_1-\frac{r}{L}e_2-k_2\mu(e_3),
\label{eqn:dynamicobservererror2}\\
\dot{e}_3 &=& \frac{3}{4C}(u_d e_1+u_q e_2)-\mu(e_3).
\label{eqn:dynamicobservererror3}
\end{eqnarray}
where $\mu(e_3)$ is the same as in equation (\ref{eqn:stw1}).

The design problem is transformed into determining $\alpha,\lambda$ and $k_1,k_2$
which are the tuning parameters to
ensure the convergence of the error system
(\ref{eqn:dynamicobservererror1},\ref{eqn:dynamicobservererror2},
\ref{eqn:dynamicobservererror3}).
\begin{proposition}\label{thm:2}
Consider the system (\ref{eqn:dynamicobservererror3}),
and assume that the control inputs are bounded,
\begin{equation}\label{ineqn:bounded condition}
    \|\bm{u_{dq}}\|_2\leq \sqrt{2},
\end{equation}
Then, the trajectories of the system (\ref{eqn:dynamicobservererror3}) converge
to zero in finite time, and the resulting reduced order
dynamics (\ref{eqn:dynamicobservererror1},\ref{eqn:dynamicobservererror2})
are exponentially stable,
if the gains $\alpha,\lambda$ of the STA
and tuning parameters $k_1,k_2$ are chosen as \cite{levant1998robust},
\begin{equation}\label{ineqn:alpha and ramada1}
\begin{split}
    \alpha >& F, \quad
    \lambda^2>\alpha,
\end{split}
\end{equation}
\begin{equation}\label{ineqn:k1k2}
\begin{split}
k_1=&
\begin{cases}
  \kappa u_d, \ \ \quad if \ \ |e_3|=0\\
  0. \ \ \quad\quad else
\end{cases},\\
k_2=&
\begin{cases}
  \kappa u_q, \ \ \quad if \ \ |e_3|=0\\
  0. \ \ \quad\quad else
\end{cases}
\end{split}
\end{equation}
where $F$ and $\kappa$ are some positive constants.
\end{proposition}
\begin{proof}\label{eqn:finite time proof}
The proof is divided into two steps. In the first step, the equation
(\ref{eqn:dynamicobservererror3}) is proven to be finite time stable.
Then, the resulting reduced order
dynamics (\ref{eqn:dynamicobservererror1},\ref{eqn:dynamicobservererror2})
are proven to be exponentially stable with faster convergence rate
than its open loop dynamics (detectability property).
At the beginning, the two correction gains $k_1,k_2$ are zero due to $e_3\neq 0$.
The system (\ref{eqn:dynamicobservererror1},\ref{eqn:dynamicobservererror2}) becomes
\begin{equation}\label{eqn:open loop obs}
\begin{bmatrix}
\dot{e}_1\\
\dot{e}_2
\end{bmatrix}
=
\underbrace{\begin{bmatrix}
-\frac{r}{L}&\omega\\
-\omega&-\frac{r}{L}
\end{bmatrix}}_{A}
\begin{bmatrix}
{e}_1\\{e}_2
\end{bmatrix}
\end{equation}
It is easy to conclude that the system (\ref{eqn:open loop obs})
is exponentially stable given that $A$ is a Hurwitz matrix.
Consequently, we can conclude that $e_1,e_2,\dot{e}_1,\dot{e}_2$ are bounded
\begin{equation}\label{eqn:bounded condition1}
\begin{split}
|e_1(t)|\leq& |e_1(0)|, \ \ |e_2(t)|\leq |e_2(0)|,\\
|\dot{e}_1(t)|\leq& \frac{r}{L}|e_1(t_0)|+\omega |e_2(t_0)|, \\
|\dot{e}_2(t)|\leq& \omega |e_1(t_0)|+\frac{r}{L}|e_2(t_0)|.
\end{split}
\end{equation}

The condition of (\ref{ineqn:bounded condition}) is deduced from the Park
transformation (\ref{transformed variables}) and the control inputs
$\bm{u}$ which take values from the discrete set $\{-1,+1\}$,
\begin{equation}\label{eqn:bounded input}
\begin{split}
\|\bm{u_{dq}}\|_2\leq& \|T\|_2\|\bm{u}\|_2= \sqrt{\lambda_{max}(T^TT)}\|\bm{u}\|_2\\
\leq&\sqrt{\frac{2}{3}}\sqrt{3}=\sqrt{2},
\end{split}
\end{equation}
The equation (\ref{eqn:dynamicobservererror3}) can be rewritten as,
\begin{equation}
\left\{
\begin{split}
\dot{e}_3=& -\lambda |e_3|^{\frac{1}{2}}sign(e_3) + \varphi,\\
\dot{\varphi}=&-\alpha sign(e_3) + g(e_1,e_2,u_d,u_q),
\end{split}
\right.
\end{equation}
where $g(e_1,e_2,u_d,u_q)=\frac{3}{4C}\frac{d}{dt}(u_d e_1(t)+u_q e_2(t))$ is considered as
a bounded decreasing perturbation.

It follows from (\ref{eqn:bounded condition1}, \ref{eqn:bounded input})
that $\left\vert g \right\vert \leq F$, with a positive value $F$.
Given that the gains of the STA are chosen as (\ref{ineqn:alpha and ramada1}),
$e_3,\dot{e}_3$ converge to zero in finite time \citep{levant1998robust}.

Thereafter, the equivalent output-error injection $\mu(e_3)$
in (\ref{eqn:dynamicobservererror3})
can be obtained directly without any low pass filters,
\begin{equation}\label{eqn:equivalent injection1}
    \mu(e_3)=\frac{3}{4C}(u_d e_1+u_q e_2),
\end{equation}
When the sliding motion takes place ($e_3=0,\dot{e}_3=0$), the gains $k_1,k_2$ will switch
according to (\ref{ineqn:k1k2}).
Then, the equivalent injection (\ref{eqn:equivalent injection1}) is substituted into
the system (\ref{eqn:dynamicobservererror1},\ref{eqn:dynamicobservererror2})
in order to get the reduced order dynamics,
\begin{equation}\label{eqn:reduced order dynamics}
\begin{bmatrix}
\dot{e}_1\\
\dot{e}_2
\end{bmatrix}
=
\begin{bmatrix}
-\frac{r}{L}&\omega\\
-\omega&-\frac{r}{L}
\end{bmatrix}
\begin{bmatrix}
e_1\\
e_2
\end{bmatrix}
-\bar{A}
\begin{bmatrix}
e_1\\
e_2
\end{bmatrix}
\end{equation}
where $\bar{A}=\frac{3\kappa}{4C}
\begin{bmatrix}
u^2_d&u_du_q\\
u_du_q&u^2_q
\end{bmatrix}$.

Consider a candidate Lyapunov function for system (\ref{eqn:reduced order dynamics}) as,
\begin{equation}\label{candidate Lyapunov}
 V(e)=e^TPe,
\end{equation}
where $e^T=[e_1,e_2]$,
and $P=\begin{bmatrix}
\frac{L}{2r}&0\\
0&\frac{L}{2r}
\end{bmatrix}$ which satisfies the equation
${A}^TP+P{A}=-I$, where $I$ is an identity matrix.

Then, the time derivative of $V$ along the trajectories
of system (\ref{eqn:reduced order dynamics}) is given by,
\begin{equation}\label{derivative candidate Lyapunov}
\begin{split}
 \dot{V}(e)=&-e^Te-e^T(\bar{A}^TP+P\bar{A})e\\
 =&-e^2_1-e^2_2-\frac{3\kappa L}{4rC}(u_de_1+u_qe_2)^2\\
 \leq& -e^2_1-e^2_2,
\end{split}
\end{equation}
It should be noted from (\ref{derivative candidate Lyapunov})
that for any positive $\kappa$,
the system (\ref{eqn:reduced order dynamics})
is exponentially stable with faster convergence rate
compared with the open loop dynamic (\ref{eqn:open loop obs}).
The proof of the Proposition \ref{thm:2} is finished.
\end{proof}

In the next Subsection, an output feedback ST current control is designed in order
to achieve the objective of unity power factor
and ripple free output voltage.
\subsection{Output Feedback ST Sliding Mode Current Control}
\label{subsec:sliding mode current control}
The STA is popular among the Second Order Sliding Mode (SOSM) algorithms
because it is a unique absolutely continuous sliding mode algorithm,
therefore it does not suffer from the problem of
chattering  \citep{levant1993sliding,levant2007principles}.
The main advantages of the ST SMC \citep{acdc_converter} are as follows:
\begin{enumerate}
  \item It does not need the evaluation of the time derivative of the sliding variable;
  \item Its continuous nature suppresses arbitrary disturbances with bounded time derivatives;
\end{enumerate}
The control objectives are define in the Subsection \ref{sec:Objectives}.
\subsubsection{Desired current calculation with unity power factor}\label{subsec:referencecurrent}
Normally, the value of the inductance $L\ll1$ in the system (\ref{eqn:dqmodel}),
and the right-hand sides of the equations in (\ref{eqn:dqmodel})
have the values of the same order.
Hence $\frac{di_d}{dt}, \frac{di_q}{dt} \gg \frac{dU_0}{dt}$,
implying that the dynamics of $i_d$ and $i_q$
are much faster than those of $U_{0}$ \citep{acdc_converter}.
Provided that the fast dynamics are stable,
based on the singular perturbation theory \citep{khalil1996nonlinear},
let the first and second equations of (\ref{eqn:dqmodel})
be zero formally, then $u_d,u_q$ can be obtained,
\begin{equation}\label{eqn:idqref}
\left\{
\begin{split}
\frac{U_{0}}{2L}u_d=&-\frac{r}{L}i^*_d+\omega i^{*}_q,\\
\frac{U_{0}}{2L}u_q=&-\frac{r}{L}i^*_q-\omega i^{*}_d+\frac{E}{L},\\
\frac{dU_{0}}{dt}=&-\frac{U_{0}}{R_LC}+\frac{3}{4C}(i^{*}_du_d+i^{*}_qu_q).
\end{split}
\right.
\end{equation}
Based on equation (\ref{eqn:idqref}), the reference currents $i^*_d,i^*_q$
will be determined depending on the desired system performance.
Substitute the first and second equation
into the third equation of (\ref{eqn:idqref}) yields,
\begin{equation}\label{eqn:U0}
    \frac{dU_{0}}{dt}=-\frac{U_{0}}{R_LC}+
    \frac{3}{2}\frac{i^{*}_qE-E_{r}}{U_{0}C},
\end{equation}
where $E_{r}=r({i^{*}_d}^{2}+{i^{*}_q}^{2})$ represents
the energy consumed by parasitic phase resistance.

Considering the Power-Balance condition \citep{komurcugil1998lyapunov},
\begin{equation}\label{eqn:power-balance}
   \frac{3}{2}(i^{*}_qE-E_r)=\frac{U^{*2}_0}{R_L},
\end{equation}
The reference current $i^{*}_d$ is set to zero for guaranteeing
\emph{Unity-Power-Factor}  which leads to the following calculation
of the reference current $i^*_q$,
\begin{equation}\label{eqn:iqreference}
    i^*_q=\frac{E}{2r}\pm\frac{1}{2}\sqrt{\frac{E^2}{r^2}-\frac{8U^{*2}_0}{3R_Lr}},
\end{equation}
under constraint for desired output voltage $U^*_0$,
\begin{equation}\label{eqn:voltage constraint}
    U^*_0\leq E \sqrt{\frac{3R_L}{8r}},
\end{equation}

Finally, $i^{*}_d$ and $i^{*}_q$ are obtained
as follows (due to minimal energy consumption),
\begin{equation}\label{eqn:refercurrents}
\left\{
\begin{split}
i^{*}_d=&0,\\
i^*_q=&\frac{E}{2r}-\frac{1}{2}\sqrt{\frac{E^2}{r^2}-\frac{8U^{*2}_0}{3R_Lr}},.
\end{split}
\right.
\end{equation}
As tracking error vector approaches zero, i.e., $i_d\rightarrow i^*_d$
and $i_q\rightarrow i^*_q$, the zero dynamics have the form \cite{1187319}
\begin{equation}\label{eqn:dynamicofvoltage}
    \frac{dU_{0}}{dt}=-\frac{U_0}{R_LC}+\frac{U^{*2}_0}{R_LCU_0},
\end{equation}
Define a new variable $Z=U^2_0$,
equation (\ref{eqn:dynamicofvoltage}) can be rewritten as,
\begin{equation}\label{eqn:Z variable}
    \frac{dZ}{dt}=-\frac{2}{R_LC}(Z-U^{*2}_0),
\end{equation}
For a positive initial value of the output voltage, the steady-state value of $U_0$
will converge to the desired level $U^*_0$ with
the time constant $\frac{R_LC}{2}$ exponentially.
Therefore, the tracking of the reference current achieves the regulation of output voltage to the
desired value $U^*_0$ with a unity power factor.
In the following Subsection \ref{subsub:current control},
the design of current control based on STA is proposed.
\subsubsection{Output Sliding Mode Current Control}\label{subsub:current control}
We will now design ST SMC for the system (\ref{eqn:dqmodel}) based
on the proposed observer (\ref{eqn:dqobservermodel}).
The switching variables for the current control are defined as,
\begin{equation}\label{eqn:switchingfuction}
\left\{
\begin{split}
s_d=&i_{d}^{*}-{i}_d=i_{d}^{*}-{\hat{i}}_d+\hat{i}_d-{i}_d=\hat{s}_d-e_1,\\
s_q=&i_{q}^{*}-{i}_q=i_{q}^{*}-{\hat{i}}_q+\hat{i}_q-{i}_q=\hat{s}_q-e_2.
\end{split}
\right.
\end{equation}
where $i^{*}_{d}$ and $i^{*}_{q}$ are the desired
values of the currents in the $(d,q)$ coordinate frame
and $e_1,e_2$ are observation errors defined in (\ref{eqn:stw1}).
The desired value is selected to provide the DC power balance
between the input power and the output power.

Taking the first time derivative of $\bm{s_{dq}}=[s_d,s_q]^T$ yields,
\begin{equation}\label{timederivativeS}
\dot{\bm{s}}_{\bm{dq}}=
\begin{bmatrix}
\frac{r}{L}{i}_d-\omega {i}_q\\
\dot{i}^{*}_{q}+\frac{r}{L}{i}_q-\frac{E}{L}+\omega {i}_d
\end{bmatrix}
+\frac{U_{0d}}{2L}
\begin{bmatrix}
u_d\\u_q
\end{bmatrix}
,
\end{equation}

The equation (\ref{timederivativeS}) can be rewritten as,
\begin{equation}\label{eqn:supertwisting}
\begin{split}
\begin{bmatrix}
\dot{s}_d\\
\dot{s}_q
\end{bmatrix}
=&
\begin{bmatrix}
-\frac{r}{L}s_d+\omega s_q\\
-\omega s_d-\frac{r}{L}s_q
\end{bmatrix}
+
\begin{bmatrix}
F_d\\
F_q
\end{bmatrix}
+\frac{U_{0d}}{2L}
\begin{bmatrix}
u_d\\
u_q
\end{bmatrix}
,
\end{split}
\end{equation}
where $F_d=-\omega{i}^*_q$,
$F_q=\dot{i}^*_q+\frac{r}{L}{i}^*_q-\frac{E}{L}$.

The control objective is to force the sliding variable $s_d,s_q$
to zero. Design  controls $u_d,u_q$ as follows,
\begin{equation}\label{eqn:stw_controller1}
\begin{split}
\bm{{u}_{dq}}=&
\begin{bmatrix}
{u}_d\\
{u}_q
\end{bmatrix}
=\frac{2L}{U_{0d}}
\begin{bmatrix}
\frac{r}{L}\hat{s}_d-\omega \hat{s}_q-\mu(\hat{s}_d)-F_d\\
\omega \hat{s}_d+\frac{r}{L}\hat{s}_q-\mu(\hat{s}_q)-F_q
\end{bmatrix},
\end{split}
\end{equation}
where $\mu(\hat{s}_d)$ and $\mu(\hat{s}_q)$ take the form of 
(\ref{eqn:stw1}).
\begin{remark}
Since the control signals (\ref{eqn:stw_controller1}) are continuous,
Pulse Width Modulation (PWM) technique is used to implement
the controller on practical platforms \citep{sabanovic_variable_2004}.
\end{remark}

Then, Lyapunov analysis is used to prove the convergence of the
system (\ref{eqn:supertwisting}) under the controller
(\ref{eqn:stw_controller1}) where observation errors are taken into account.
\begin{theorem}\label{thm:3}
The output ST SMC
(\ref{eqn:stw_controller1})
ensures the exponential convergence of the state trajectories of
the system (\ref{eqn:supertwisting}) to the origin $s_d=0$, $s_q=0$,
if the gains of STA $\mu(\hat{s}_d),\mu(\hat{s}_q)$ satisfy
the following conditions \citep{levant1998robust},
\begin{equation}\label{ineqn:alpha and ramada}
\begin{split}
    \alpha_d >&F_d, \quad  \lambda^2_d > \alpha_d,\\
    \alpha_q >&F_q, \quad  \lambda^2_q > \alpha_q,
\end{split}
\end{equation}
where $F_d$ and $F_q$ are some positive constants.
\end{theorem}
\begin{proof}
Substitute the control (\ref{eqn:stw_controller1})
into system (\ref{eqn:supertwisting}) yields,
\begin{equation}\label{eqn:new form}
\left\{
\begin{split}
\dot{\hat{s}}_{d}=&-\mu(\hat{s}_d)-\frac{r}{L}e_1+\omega e_2+\dot{e}_1,\\
\dot{\hat{s}}_{q}=&-\mu(\hat{s}_q)-\omega e_1 - \frac{r}{L}e_2 +\dot{e}_2.
\end{split}
\right.
\end{equation}
where $e_{12}^T=\begin{bmatrix}
e_1&e_2
\end{bmatrix}$
are the observation errors.
From the result of Proposition \ref{thm:2},
$\lim_{t\rightarrow \infty}e_1(t)=0$ and
$\lim_{t\rightarrow \infty}e_2(t)=0$.
It follows from (\ref{eqn:bounded condition1}, \ref{eqn:reduced order dynamics}) that
the time derivatives of the terms $-\frac{r}{L}e_1+\omega e_2+\dot{e}_1$
and $-\omega e_1 - \frac{r}{L}e_2 +\dot{e}_2$ are bounded,
\begin{equation}\label{eqn:bounded condition}
\begin{split}
  &\left\vert \frac{d}{dt}\left(-\frac{r}{L}e_1+\omega e_2+\dot{e}_1\right)\right\vert \leq F_d,\\
  &\left\vert\frac{d}{dt}\left(-\omega e_1 - \frac{r}{L}e_2 +\dot{e}_2\right)\right\vert \leq F_q.
\end{split}
\end{equation}
where $F_d, F_q$ are some positive constants.

From the result of \cite{levant1998robust},
the trajectories of the system (\ref{eqn:new form}) converge
to zero $\hat{s}_d=0,\hat{s}_q=0$
in finite time under the gains chosen as (\ref{ineqn:alpha and ramada}).
It can be concluded that
$s_d,s_q$ converge to zero exponentially since
$s_d,s_q$ converge to $\hat{s}_d,\hat{s}_q$ exponentially.
Thus, Theorem \ref{thm:3} is proven.
\end{proof}

The proposed observer-based control law (\ref{eqn:stw_controller1}) requires
real-time evaluation of sliding variables $\hat{s}_d=i^*_d-\hat{i}_d,\hat{s}_q=i^*_q-\hat{i}_q$.
However, the current reference $i^*_q$ in (\ref{eqn:refercurrents})
requires the knowledge of load resistance $R_L$ and
parasitic phase resistance $r$.
Due to this fact, a ST parameter observer is employed to
estimate the value of load resistance while phase resistance
is assumed to have nominal value \citep{unitypowerfactor}.
\section{ST PARAMETER OBSERVER DESIGN AND POWER FACTOR ESTIMATION}\label{sec:ParameterObservers}
\subsection{Load Resistance Estimation}
In this work, the load resistance $R_L$ in the system is assumed to vary around its
nominal value $R_0$.
The last differential equation in (\ref{eqn:dqmodel})
is used to construct the observer dynamics using
ST sliding mode technique
\begin{equation}\label{eqn:loadestimation2}
\begin{split}
\frac{d\hat{U}_{0}}{dt} =& -\frac{U_{0}}{R_0C}+\frac{3}{4C}(\hat{i}_du_d+\hat{i}_q u_q)
+\mu(\tilde{U}_0),
\end{split}
\end{equation}
where $R_0$ is the nominal value of the load resistance.

Consider the observation error $\tilde{U}_0=U_0-\hat{U}_0$
which has the following dynamics according to
(\ref{eqn:dqmodel}) and (\ref{eqn:loadestimation2}),
\begin{equation}\label{eqn:loadestimation3}
\begin{split}
    \dot{\tilde{U}}_0=&-\mu(\tilde{U}_0)-\frac{U_0}{C}\left(\frac{1}{R_L}-\frac{1}{R_0}\right)\\
    +&\frac{3}{4C}(e_1u_d+e_2u_q)=-\mu(\tilde{U}_0)+\Psi_{R_L},
\end{split}
\end{equation}
where $\mu(\cdot)$ is the STA defined in (\ref{eqn:stw1}).
From Proposition \ref{thm:2}, $\lim\limits_{t \to \infty}{e_1}=0$,
$\lim\limits_{t \to \infty }{e_2}=0$ and
$\lim\limits_{t \to \infty}{\Psi_{R_L}}=
-\frac{U_0}{C}\left(\frac{1}{R_L}-\frac{1}{R_0}\right)$.
Sliding mode will be enforced
with appropriate values of $\lambda,\alpha$
providing that the first time derivative of the term $\Psi_{R_L}$ is bounded
\cite{levant1993sliding,levant2007principles}.
It follows that when a sliding motion takes place,
\begin{equation}\label{eqn:sliding motion}
-\mu(\tilde{U}_0)-\frac{U_0}{C}\left(\frac{1}{R_L}-\frac{1}{R_0}\right)
    +\frac{3}{4C}(e_1u_d+e_2u_q)=0,
\end{equation}
The load resistance $R_L$ can therefore be estimated in terms of its
nominal value and observer's output with appropriate parameters exponentially,
\begin{equation}\label{eqn:estimationofloadresistance}
    \hat{R}_L=\frac{R_0U_0}{U_0-R_0C\mu(\tilde{U}_0)},
\end{equation}
since
\begin{equation}\label{eqn:exponentially}
    \lim_{t\rightarrow \infty}\|\hat{R}_L-R_L\|=0,
\end{equation}
\subsection{Power Factor Estimation}\label{sec:PFC}
The estimation of power factor value is very important for analyzing the quality
of the proposed control law.
The definition of power factor is given as the following formula,
\begin{equation}\label{eqn:powerfactor}
    PF=PF_{h}\cdot PF_{d}=
    \frac{RMS(i_{1}(t))}{RMS(i(t))}\cdot cos(\phi),
\end{equation}
where $PF_{h}$ is the harmonic distortion and $PF_{d}$ is the displacement
between input phase current and source voltage.
The ideal condition of \emph{Unity Power Factor} corresponds to no
harmonic distortion(phase current has only main harmonic) and
no phase shift between input phase current and main source voltage \citep{unitypowerfactor}.

The $RMS(\cdotp)$ stands for the root-mean-square quantity which
is calculated as follows,
\begin{equation}\label{root-mean-square}
RMS(i(t))=\sqrt{\frac{1}{T}\int^T_0i^2(\tau)d\tau}
\end{equation}
where $T$ is the period of the phase current $i(t)$,
$RMS(i_1(t))$ characterizes the fundamental component of the current (root-mean-square)
and $RMS(i(t))$ corresponds to the total current (root-mean-square).

The overall power factor for three-phase AC/DC converter
is calculated from the estimates of phase current which
will be a product of the three single phase power factor values,
\begin{equation}\label{eqn:totalpowerfactor}
\begin{split}
PF_{total}=&PF_{1}\cdot PF_{2}\cdot PF_{3},
\end{split}
\end{equation}

The structure of estimation for the single-phase power factor
in MATLAB/SIMULINK is shown in Fig. \ref{fig:singlepf},
which includes two important modules:
\emph{Fourier analysis} and \emph{Harmonic analysis}.
The first module gives the phases of main-frequency input current and
source phase voltage respectively.
The other module is used to measure the total harmonic distortion of input current.
\begin{figure}[!htb]
  \centering
  \includegraphics[width=4.4in]{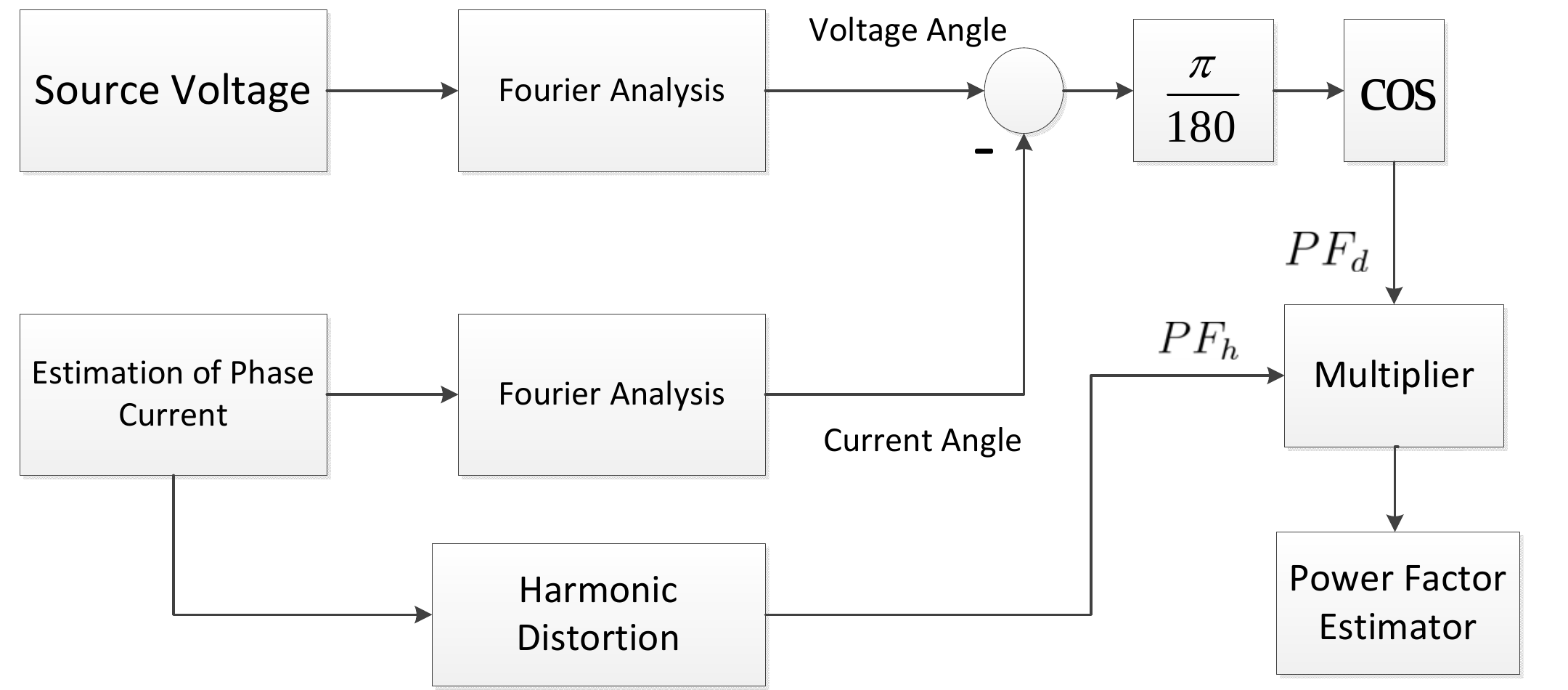}
  \caption{Structure of single-phase power factor calculation}
  \label{fig:singlepf}
\end{figure}
\section{SIMULATION RESULTS}
Simulation results are carried out on the proposed three phase
AC/DC boost power converter,
the parameters used in simulation are shown in Table \ref{table:2}.
Load resistance and frequency are varied to test controller's
ability to handle with varying conditions at time $1.0s$ and $1.5s$ respectively.
\begin{table}
\tbl{Parameters Used For Simulation.}
{\begin{tabular}{@{}llllll}
\toprule
$r$ & 0.02              & $\Omega$
&$C$   & 100               &$u$F\\
$L$   & $2$       & mH
&$w$   &$150\pi\stackrel{t=1.5s}{\longrightarrow}300\pi$  & rad/s\\
$R$ &$50\stackrel{t=1.0s}{\longrightarrow}40$        & $\Omega$
&$E$   &150       & V\\
${U_{0}}^{*}$ &650 &V
&$U_{0}(0)$ &5 &V\\
\botrule
\end{tabular}}
\label{table:2}
\end{table}

The simulation results of the proposed observer-based ST SMC compared
with linear PI regulator \citep{767067}
are shown in Figs \ref{currentandvoltage}-\ref{powerfactorcalculation}.
Input phase current along with the corresponding source voltage are shown in
Fig. \ref{currentandvoltage}.
From Fig. \ref{fig:cvfirst} and Fig. \ref{fig:cvsecond},
both of the controller make no phase shift between
the input current and corresponding source voltage,
however the PI Control results in higher harmonics compared to the ST SMC.

Fig. \ref{outputvoltage} shows the output voltage performance of the AC/DC converter.
From Fig. \ref{fig:dcvfirst} and Fig. \ref{fig:dcvsecond},
it is seen that the proposed observer-based ST SMC is able to regulate the output voltage
to the desired level under the condition of load variation.
A good estimate for the load resistance is shown in Fig. \ref{fig:estimationloadresistance}.
However, the PI Control results in higher fluctuation around some DC level,
and higher voltage overshoot compared with the ST SMC.
It should be noted that the PI control is not able to demonstrate
its robustness with respect to load variation, due to the fact
that the gains $k_p,k_i$ of the PI control depend
on the load resistance $R_L$ \citep{767067}.

Fig. \ref{powerfactorcalculation} shows the separate power factor value of
each phase and their product as a combined characteristic of the AC/DC converter.
From Fig. \ref{fig:pfsupertwisting} and
Fig. \ref{fig:pftraditional}, it is shown that the power factor value are more that 97\%
in case of observer-based ST SMC,
while the PI Control results in less value and more oscillations
compared to the ST SMC.
The proposed observer-based ST SMC is proven to be able to produce a power factor value that
was always more than 97\%, and less sensitive to the changing conditions.
\section{CONCLUSIONS}
An observer-based ST SMC is proposed in this paper for the AC/DC boost converters.
The use of observer reduces the number of current sensors, decreases the system cost,
volume and provides robustness to
the change of operational condition (e.g. in load resistance $R_L$ and
frequency of the source voltage $\omega$).
The proposed observer-based ST SMC maintains the power factor close to unity.
A strong Lyapunov function is introduced to prove the stability of the
observer and controller as a whole system.
Simulation results show that the observer-based controller performs better,
compared to conventional PI control, with less overshoot and less sensitivity
to disturbance and parametric uncertainty.
\begin{figure}[!htb]%
\begin{center}
\subfigure[The case of observer-based ST SMC]
{\label{fig:cvfirst}\includegraphics[width=2.5in]{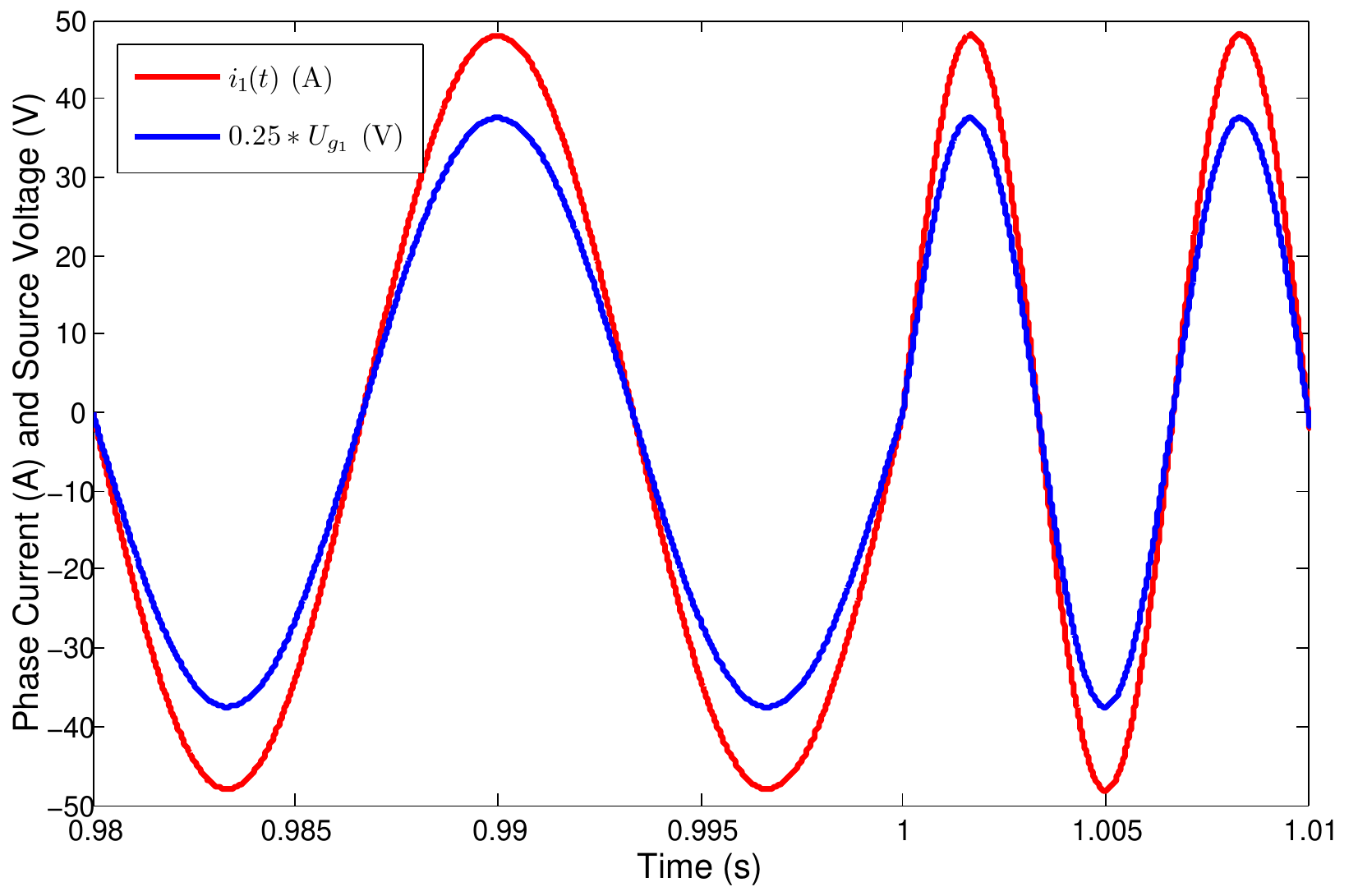}}
\quad
\subfigure[The case of PI Control]
{\label{fig:cvsecond}\includegraphics[width=2.5in]{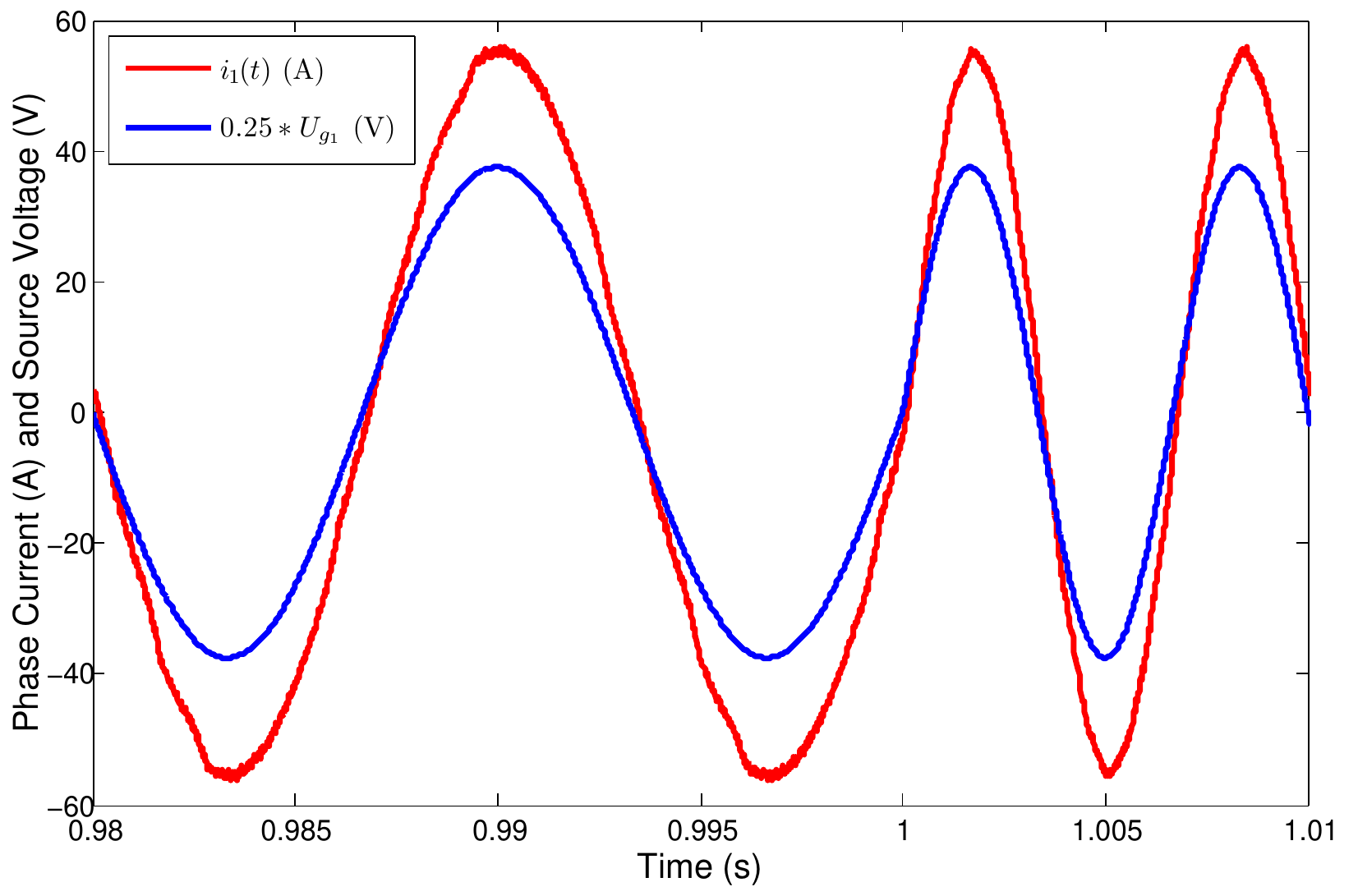}}
\caption{Phase current and source voltage ($\times0.25$)}
\label{currentandvoltage}
\end{center}
\end{figure}
\begin{figure}[!htb]%
\begin{center}
\subfigure[The case of observer-based ST SMC]
{\label{fig:dcvfirst}\includegraphics[width=2.5in]{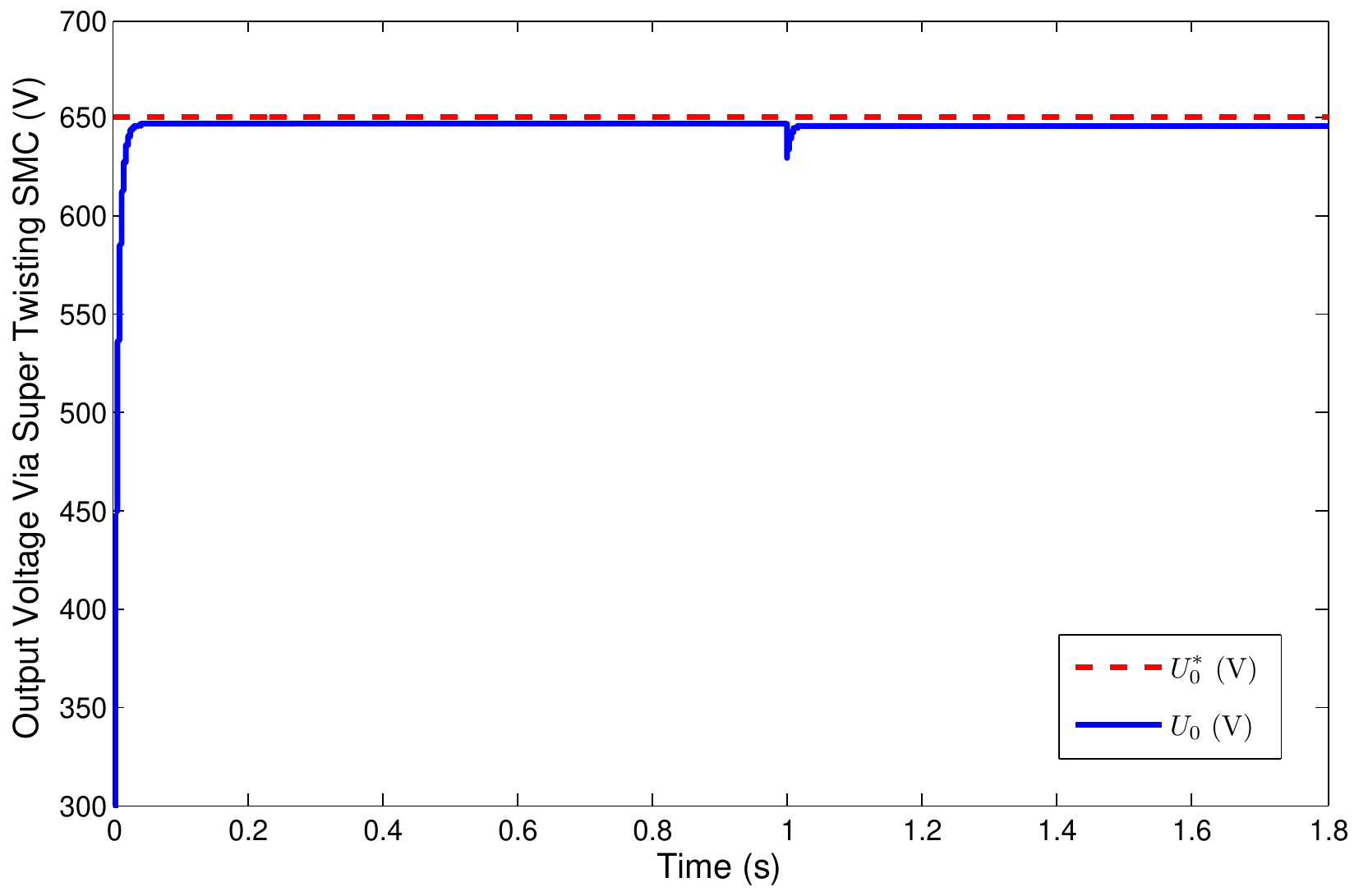}}
\quad
\subfigure[The case of PI Control]
{\label{fig:dcvsecond}\includegraphics[width=2.5in]{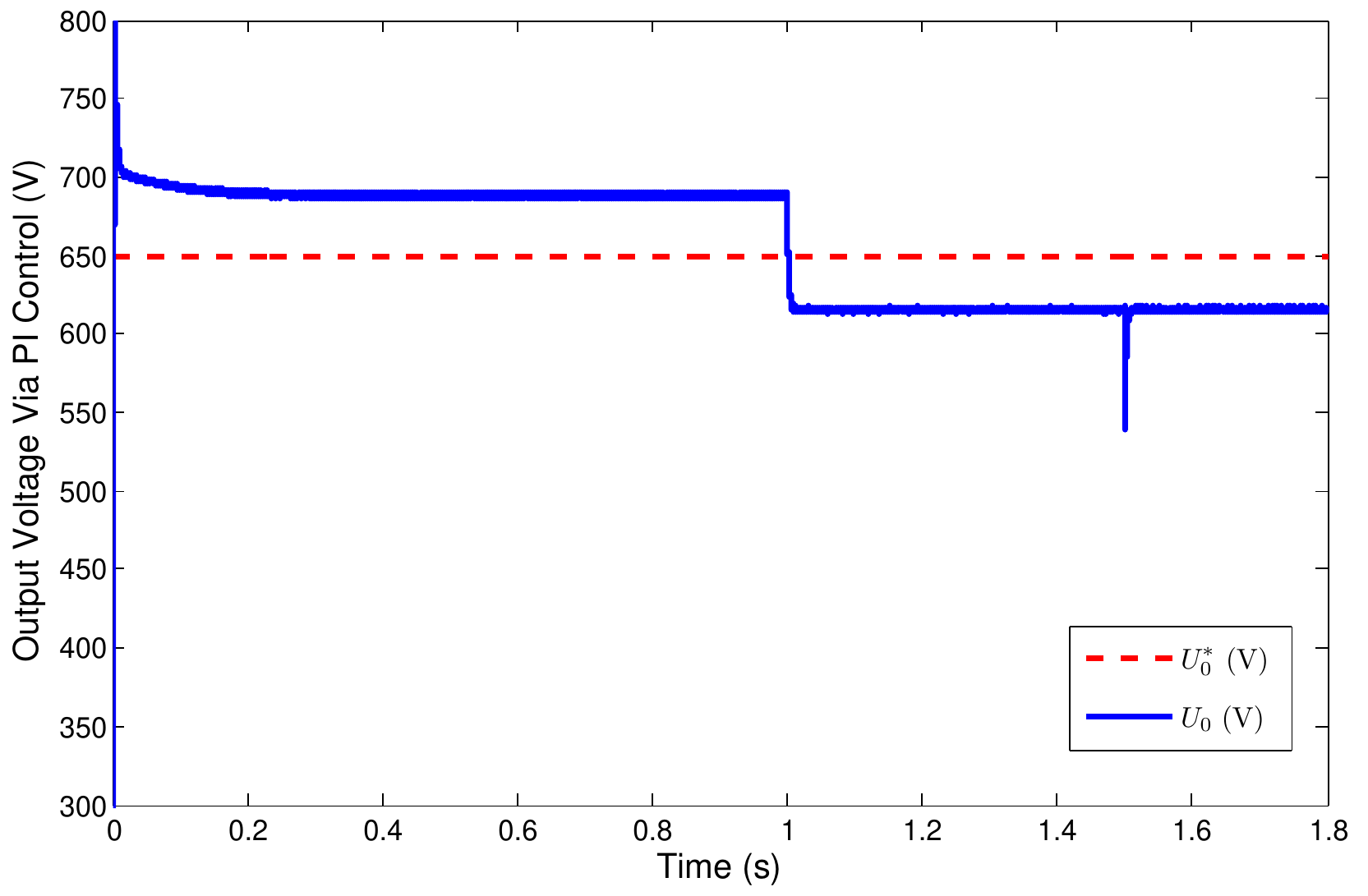}}
\caption{Output Voltage Performance}
\label{outputvoltage}
\end{center}
\end{figure}
\begin{figure}[!htb]
  \centering
  \includegraphics[width=4in]{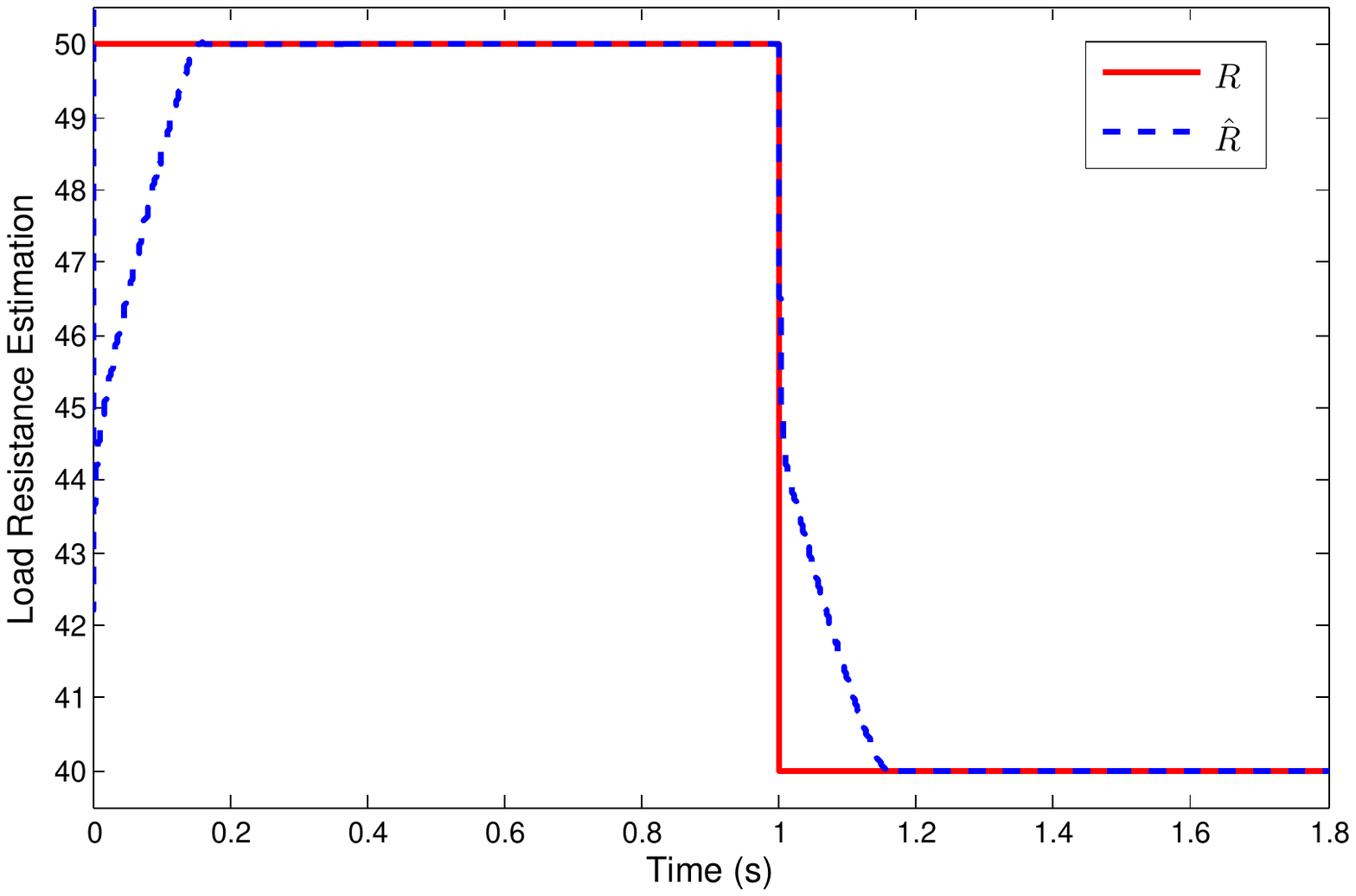}
  \caption{Load Resistance Estimation Via Super-Twisting Observer}
  \label{fig:estimationloadresistance}
\end{figure}
\begin{figure}[!htb]%
\centering
\subfigure[The case of super-twisting SMC]{\label{fig:pfsupertwisting}
\includegraphics[width=2.5in]{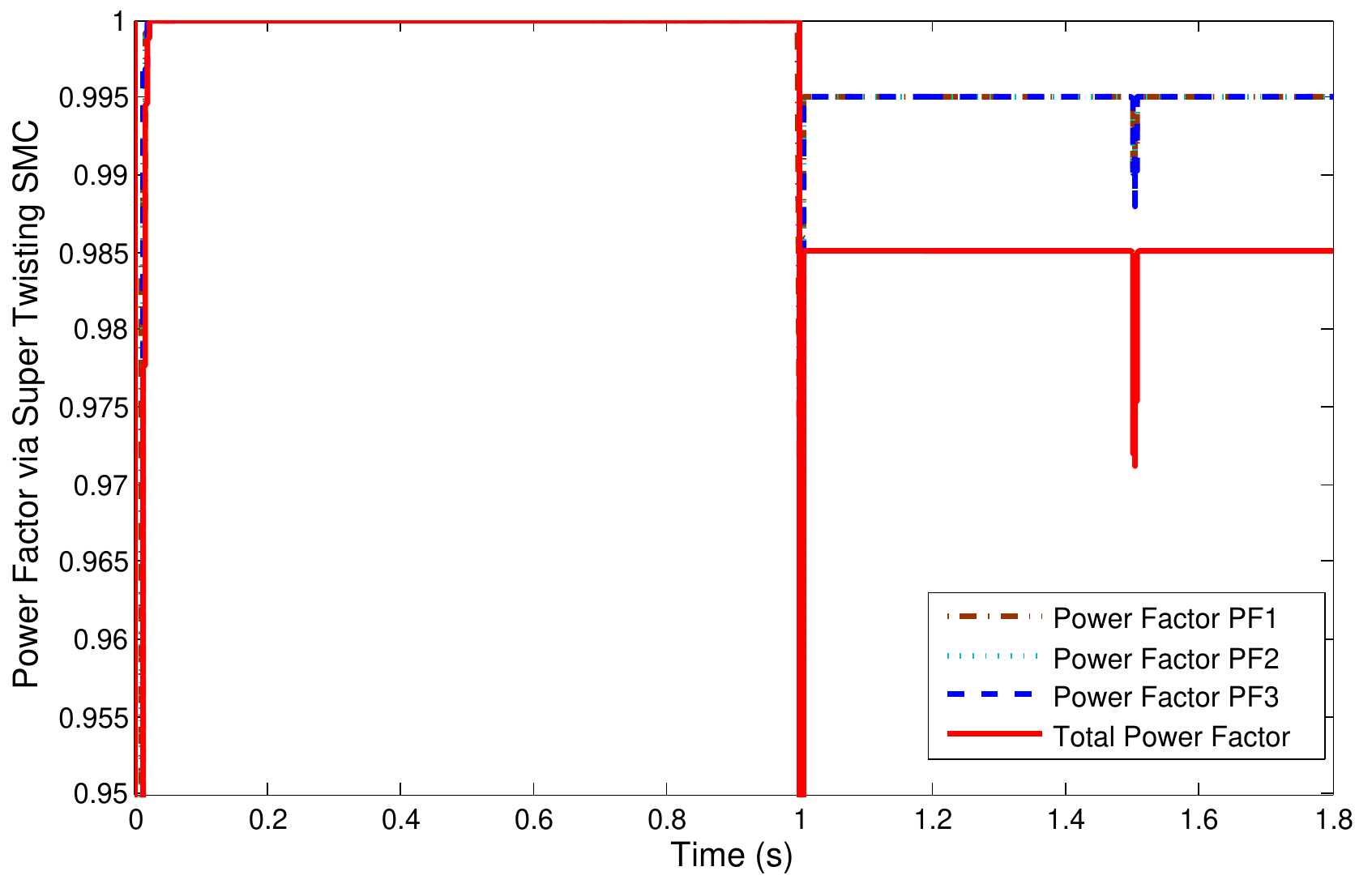}}
\quad
\subfigure[The case of PI Control]
{\label{fig:pftraditional}\includegraphics[width=2.5in]{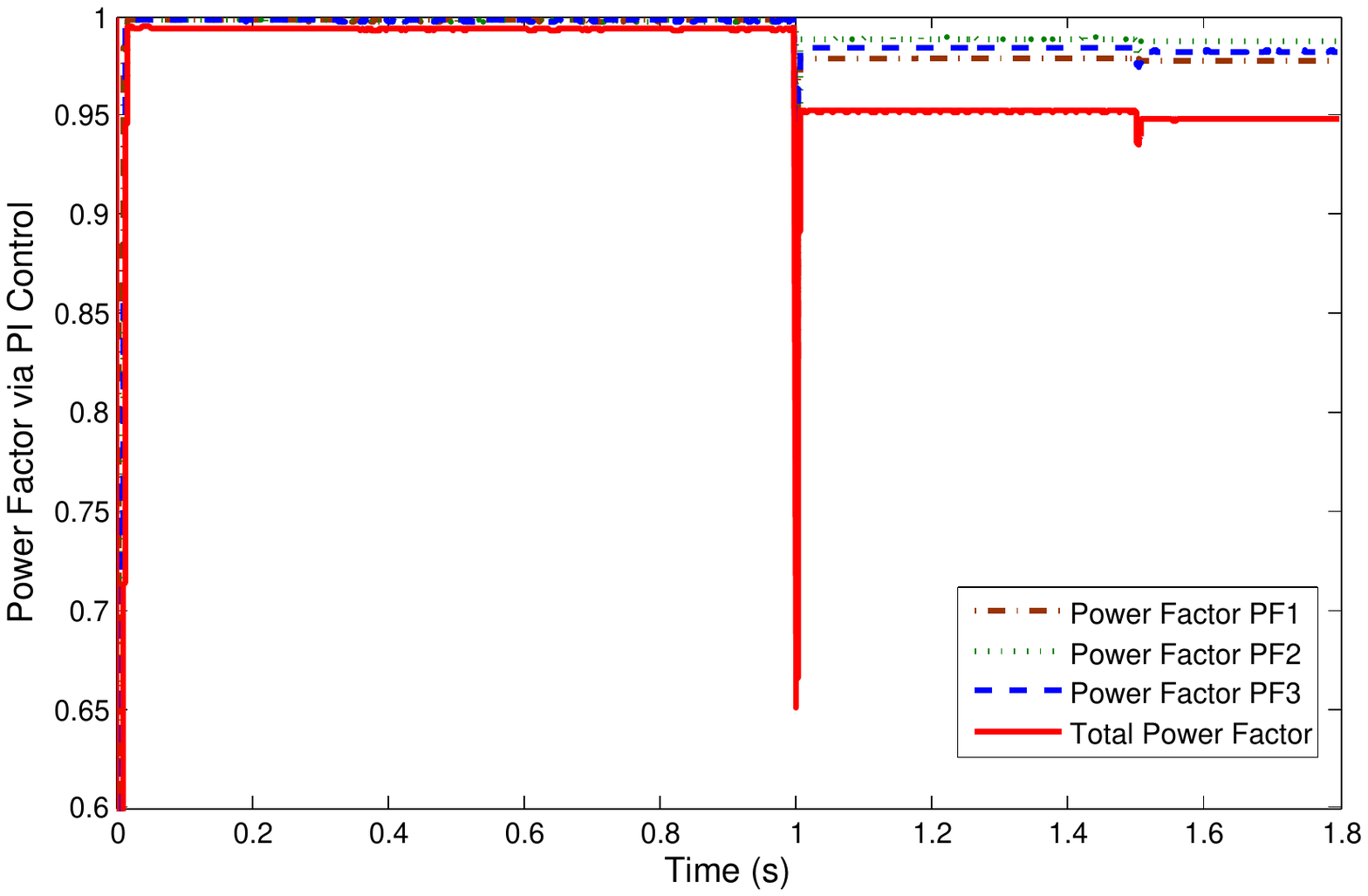}}
\caption{Power factor of the AC/DC converter}
\label{powerfactorcalculation}
\end{figure}
\newpage
%
\bibliographystyle{tCON}
\bibliography{acdc_journal1}
\end{document}